\documentclass[12pt]{article}
\usepackage{amsmath,amssymb,amsthm}
\newtheorem{lemma}{Lemma}
\newtheorem{thm}[lemma]{Theorem}
\newcommand{\C}{\mathcal{C}}

\begin{document}

\begin{center}
\Large
Tiling a unit square with 8 squares
\end{center}

\begin{flushright}
Iwan Praton \\
Franklin and Marshall College \\
Lancaster, PA 17604\\
\verb+iwan.praton@fandm.edu+
\end{flushright}

Suppose a unit square is packed with $n$ squares of side lengths
$s_1, s_2,\ldots, s_n$. We define 
$\psi_1(n)=\max \sum_{i=1}^n s_i$, where the maximum is taken
over all possible packings of the unit square.
Not a lot is known about the function $\psi_1$.
Erd{\H o}s [1] asked whether 
$\psi_1(k^2+1)=k$. More generally,
Erd{\H o}s and Soifer [2] presented
explicit packings that provided lower bounds
of $\psi_1(n)$ for all (nonsquare) $n$; they mentioned that
these lower bounds appear to be good. Thus we
have tentative values for $\psi_1(n)$.

In [3] Staton and Tyler introduced
two modifications of $\psi_1$ as follows. Define a \emph{right
packing} to be a packing by squares whose
sides are parallel to the sides of the unit square. Then $\psi_2(n)$
is defined to be $\max \sum s_i$ where the maximum is taken
over all right packings with $n$ squares. Also, $\psi_3(n)$, for
$n\neq 2, 3, 5$, is
defined to be $\max \sum s_i$, where the maximum is now taken
over all right \emph{tilings} with $n$ squares.
(A tiling is a packing where the unit square is completely filled.
The unit square can be tiled with $n$ squares for all values of
$n$ except for $n=2,3,5$, thus the restriction on $n$ in
the definition of $\psi_3$.) It is clear that
$\psi_1(n)\geq \psi_2(n)\geq \psi_3(n)$. Staton and Tyler asked for what
values of $n$ we have $\psi_1(n)=\psi_2(n)=\psi_3(n)$. 

There
are some reasons to suspect that the three functions might be
identical. The packings constructed by Erd{\H o}s and Soifer 
in [2] are actually tilings, except when $n$ differs by 1 from
a square integer. Staton and Tyler in [3] took care
of the case when $n$ is one more than a perfect square
by constructing tilings whose sums of edge lengths
are the same as the Erd{\H o}s-Soifer lower bounds.
Thus if the Erd{\H o}s-Soifer conjecture
 is correct, then 
$\psi_1(n)=\psi_2(n)=\psi_3(n)$ for all values of $n$
except possibly when $n$ is one less than a perfect square.
In
this note we show that, alas, $\psi_2(n)\neq 
\psi_3(n)$ when $n=8$; more precisely, we show
that $\psi_3(8)=2.6$. 

We first define our terminology and notation. All packings
and tilings in this paper of the unit square, so we will
omit the phrase ``of the unit square'' in what follows.
If $A$ is a square, its side length is denoted by $s_A$.
If $\mathcal{C}=\{A_1,\ldots,A_n\}$ is a collection of 
squares, we write $\sigma(\mathcal{C})=
\sigma(A_1,\dots,A_n)$ for $\sum s_{A_i}$. 

Here is an upper bound due to
Erd\"{o}s; the proof below appeared in Erd\"{o}s and Soifer [2].
\begin{lemma}\label{CS}
If $\mathcal{C}$ is a collection of $n$ squares with
total area $A$, then $\sigma(\mathcal{C})
\leq \sqrt{nA}$, with equality only if the $n$ squares
are the same size.
\end{lemma}
\begin{proof}
Let $s_1,\dots, s_n$ be the side lengths
of the $n$ squares. Apply the Cauchy-Schwarz
inequality to the $n$-component
vectors $(1,1,\dots,1)$ and $(s_1,\dots,s_n)$. 
\end{proof}

As an immediate consequence, we
see that $\psi_3(8)\leq \sqrt{8}$.
We also get a lower bound from an explicit
construction: the tiling

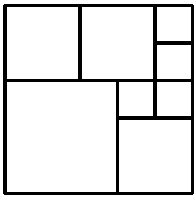
\begin{center}
\unitlength 0.5mm
\begin{picture}(50.00,50.00)(0,0)

\linethickness{0.15mm}
\put(0.00,0.00){\line(1,0){50.00}}
\put(0.00,0.00){\line(0,1){50.00}}
\put(50.00,0.00){\line(0,1){50.00}}
\put(0.00,50.00){\line(1,0){50.00}}

\linethickness{0.3mm}
\put(40.00,40.00){\line(1,0){10.00}}
\put(40.00,40.00){\line(0,1){10.00}}
\put(50.00,40.00){\line(0,1){10.00}}
\put(40.00,50.00){\line(1,0){10.00}}

\linethickness{0.3mm}
\put(40.00,30.00){\line(1,0){10.00}}
\put(40.00,30.00){\line(0,1){10.00}}
\put(50.00,30.00){\line(0,1){10.00}}
\put(40.00,40.00){\line(1,0){10.00}}

\linethickness{0.3mm}
\put(30.00,20.00){\line(1,0){10.00}}
\put(30.00,20.00){\line(0,1){10.00}}
\put(40.00,20.00){\line(0,1){10.00}}
\put(40.00,30.00){\line(1,0){10.00}}

\linethickness{0.3mm}
\put(40.00,20.00){\line(1,0){10.00}}
\put(40.00,20.00){\line(0,1){10.00}}
\put(50.00,20.00){\line(0,1){10.00}}
\put(40.00,30.00){\line(1,0){10.00}}

\linethickness{0.3mm}
\put(0.00,30.00){\line(1,0){20.00}}
\put(0.00,30.00){\line(0,1){20.00}}
\put(20.00,30.00){\line(0,1){20.00}}
\put(0.00,50.00){\line(1,0){20.00}}

\linethickness{0.3mm}
\put(20.00,30.00){\line(1,0){20.00}}
\put(20.00,30.00){\line(0,1){20.00}}
\put(40.00,30.00){\line(0,1){20.00}}
\put(20.00,50.00){\line(1,0){20.00}}

\linethickness{0.3mm}
\put(30.00,0.00){\line(1,0){20.00}}
\put(30.00,0.00){\line(0,1){20.00}}
\put(50.00,0.00){\line(0,1){20.00}}
\put(30.00,20.00){\line(1,0){20.00}}

\linethickness{0.3mm}
\put(0.00,0.00){\line(1,0){30.00}}
\put(0.00,0.00){\line(0,1){30.00}}
\put(30.00,0.00){\line(0,1){30.00}}
\put(0.00,30.00){\line(1,0){30.00}}

\end{picture}
\end{center}
shows that $\psi_3(8)\geq 2.6$. 
To show that $\psi_3(8)=2.6$, 
we need to investigate the actual
tiling in more detail.

Put our unit square so its corners are at
$(0,0)$, $(1,0)$, $(0,1)$, and $(1,1)$.
Let $\C$ be any tiling of this square
with 8 tiles.
For any $c$ where $0< c < 1$,
we define $\C_c$ to be the set of
tiles whose interior intersect 
the vertical line $x=c$. We want
to avoid the case where there is
a tile with a vertical edge on the line
$x=c$ (such a line is called
\emph{ambiguous} by Staton and
Tyler [2]), so we will assume
forthwith that the vertical line
$x=c$ is not ambiguous. Thus
$\sigma(\C_c)=1$. Note
that there is an unambiguous line
as close as we want to an
ambiguous line.

The values $c=0$ and $c=1$
are special. We call the line $x=0$ 
the \emph{left coast} and the line $x=1$ the
\emph{right coast}. The \emph{left coastal tiles} 
$\C_0$ are 
the tiles that have an edge on the left coast. 
Similarly, the \emph{right coastal tiles} $\C_1$
are those tiles with an edge
on the right coast. Their union is the
set of \emph{coastal tiles}. Tiles that
are not coastal tiles are called \emph{inland
tiles}. There are not too many of these.

\begin{lemma}\label{inlandtiles}
The sum of the side lengths of all inland
tiles is less than 1.
\end{lemma}
\begin{proof}
For any tiling $\C$, we know that
$\sigma(\C)\leq \sqrt{8}<3$. We have
$\sigma(\C_0)=\sigma(C_1)=1$. If the sum of
the side lengths of inland tiles is 1
or more, then $\sigma(\C)\geq 1+1+1
=3$, a contradiction. 
\end{proof}

\begin{lemma}\label{noinlines}
For any $0<c<1$, the set $\C_c$
contains at least one coastal tile.
\end{lemma}
\begin{proof}
Otherwise $\C_c$ contains only
inland tiles. Since $\sigma(C_c)=1$,
this contradicts Lemma~\ref{inlandtiles}.
\end{proof}

\begin{lemma}\label{twobigtiles}
There is a tile $A\in \C_0$ and $B\in \C_1$ such
that $s_A+s_B=1$. 
\end{lemma}
\begin{proof}
Let $a$ denote the maximum side lengths of
all left coastal tiles; similarly, let $b$ denote
the maximum side length of all right coastal tiles.
If $a+b<1$, then there exists a value $x_0$
(where $a<x_0<1-b$) such that the line
$x=x_0$ does not intersect any coastal tiles. 
This contradicts
Lemma~\ref{noinlines}. Thus $a+b=1$, which
is what we want.
\end{proof}
\noindent \textbf{Note}: the proof works just as well when
we turn the tiling 90 degrees. Thus there
exist two tiles, one with an edge on the line
$y=0$, and one with an edge on the line
$y=1$, such that the total edge lengths of
these two tiles is 1.

Suppose as in Lemma~\ref{twobigtiles}
we have tiles $A\in \C_0$ and $B\in \C_1$
with $s_A+s_B=1$. 
\begin{lemma}
One (or both) of $A$ and $B$ is a corner
tile.
\end{lemma}
\begin{proof}
Suppose not. 

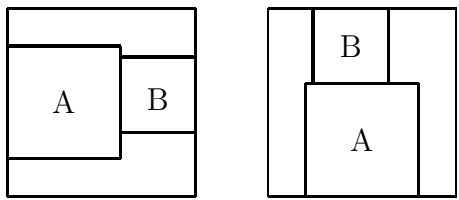
\begin{center}
\unitlength 0.5mm
\begin{picture}(50.00,50.00)(0,0)

\linethickness{0.15mm}
\put(0.00,0.00){\line(1,0){50.00}}
\put(0.00,0.00){\line(0,1){50.00}}
\put(50.00,0.00){\line(0,1){50.00}}
\put(0.00,50.00){\line(1,0){50.00}}

\linethickness{0.3mm}
\put(30.00,17.00){\line(1,0){20.00}}
\put(30.00,17.00){\line(0,1){20.00}}
\put(50.00,17.00){\line(0,1){20.00}}
\put(30.00,37.00){\line(1,0){20.00}}

\linethickness{0.3mm}
\put(00.00,10.00){\line(1,0){30.00}}
\put(00.00,10.00){\line(0,1){30.00}}
\put(30.00,10.00){\line(0,1){30.00}}
\put(00.00,40.00){\line(1,0){30.00}}

\put(15.00,25.00){\makebox(0,0)[cc]{A}}

\put(40.00,26.00){\makebox(0,0)[cc]{B}}

\end{picture}
\qquad
\unitlength 0.5mm
\begin{picture}(50.00,50.00)(0,0)

\linethickness{0.15mm}
\put(0.00,0.00){\line(1,0){50.00}}
\put(0.00,0.00){\line(0,1){50.00}}
\put(50.00,0.00){\line(0,1){50.00}}
\put(0.00,50.00){\line(1,0){50.00}}

\linethickness{0.3mm}
\put(12.00,30.00){\line(1,0){20.00}}
\put(12.00,30.00){\line(0,1){20.00}}
\put(32.00,30.00){\line(0,1){20.00}}
\put(12.00,50.00){\line(1,0){20.00}}

\linethickness{0.3mm}
\put(10.00,0.00){\line(1,0){30.00}}
\put(10.00,0.00){\line(0,1){30.00}}
\put(40.00,00.00){\line(0,1){30.00}}
\put(10.00,30.00){\line(1,0){30.00}}

\put(25.00,15.00){\makebox(0,0)[cc]{A}}

\put(22.00,40.00){\makebox(0,0)[cc]{B}}

\end{picture}
\end{center}
Then rotating the tiling 90 degrees produces
two inland tiles whose side lengths add up
to 1, contradicting Lemma~\ref{inlandtiles}.
\end{proof}

From now on we will assume, without
loss of generality, that the left
coastal tile $A$ is a corner tile, with
a corner at $(0,0)$. 

Note that we can further assume, without loss of
generality, that $B$ is also a
corner tile, with a corner at $(1,0)$.
For any tiling where $B$ has 
a corner at $(1,b)$, with $b>0$,
there is a similar tiling, with the
same total edge length, where $B$
has a corner at $(1,0)$.

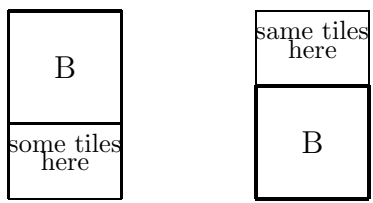
\begin{center}
\unitlength 0.5mm
\begin{picture}(30.00,50.00)(0,0)

\linethickness{0.15mm}
\put(0.00,0.00){\line(1,0){30.00}}
\put(0.00,0.00){\line(0,1){50.00}}
\put(30.00,0.00){\line(0,1){50.00}}
\put(0.00,50.00){\line(1,0){30.00}}

\linethickness{0.3mm}
\put(0.00,20.00){\line(1,0){30.00}}
\put(0.00,20.00){\line(0,1){30.00}}
\put(30.00,20.00){\line(0,1){30.00}}
\put(0.00,50.00){\line(1,0){30.00}}

\put(15.00,15.00){\makebox(0,0)[cc]{\footnotesize{some tiles}}}

\put(15.00,10.00){\makebox(0,0)[cc]{\footnotesize{here}}}

\put(15.00,35.00){\makebox(0,0)[cc]{B}}
\end{picture}
\qquad \qquad
\begin{picture}(30.00,50.00)(0,0)

\linethickness{0.15mm}
\put(0.00,0.00){\line(1,0){30.00}}
\put(0.00,0.00){\line(0,1){50.00}}
\put(30.00,0.00){\line(0,1){50.00}}
\put(0.00,50.00){\line(1,0){30.00}}

\linethickness{0.3mm}
\put(0.00,0.00){\line(1,0){30.00}}
\put(0.00,0.00){\line(0,1){30.00}}
\put(30.00,0.00){\line(0,1){30.00}}
\put(0.00,30.00){\line(1,0){30.00}}

\put(15.00,45.00){\makebox(0,0)[cc]{\footnotesize{same tiles}}}

\put(15.00,40.00){\makebox(0,0)[cc]{\footnotesize{here}}}

\put(15.00,15.00){\makebox(0,0)[cc]{B}}
\end{picture}
\end{center}

Clearly it does no harm to assume
that $s_A\geq s_B$. (Simply reflect
the tiling across the line $x=1/2$ if
necessary.) 
Thus our tiling contains a big tile
$A$, with a corner at $(0,0)$,
where $s_A\geq 1/2$. There is
also a tile $B$, with a corner
at $(1,0)$, where $s_B=1-s_A$.
Similarly (see the note after
Lemma~\ref{twobigtiles})
there is tile $B'$, with $s_{B'}=1-s_A$,
which we can assume has a corner
at $(0,1)$. 
This is enough to show that $\psi_3(8)$
is not equal to $\psi_2(8)$.

\begin{thm}
$\psi_2(8)>\psi_3(8)$. 
\end{thm}
\begin{proof}
In the standard $3\times 3$ tiling, remove
one tile. We then have a packing with 8 squares
with total edge length $\frac83$. 
Thus $\psi_2(8)\geq \frac83$, so 
all we need to show is that $\psi_3(8)<\frac83$.

Let $t=s_B$. The three tiles $A$, $B$, $B'$
have total area $2t^2+(1-t)^2=1-2t+3t^2$, leaving
an area of $2t-3t^2$ for the remaining 5
tiles. By Lemma~\ref{CS}, the total side lengths of these 5 tiles
is at most $\sqrt{5(2t-3t^2)}=\sqrt{10t-15t^2}$.
Thus the total side lengths of all 8
tiles is at most $1+t+\sqrt{10t-5t^2}$. It is straightforward
to verify that this function has a maximum
at $t=5/12$, with a maximum value
of $8/3$. Thus we get that $\psi_3(8)\leq 8/3$.

Equality is achieved only if $t=5/12$ and
the 5 tiles are all the same size. Let us figure
out what this size is. The 5 tiles have a total
area of $2t-3t^2=2\cdot (5/12)-3\cdot (5/12)^2
=45/144$, so each tile has area $9/144$,
i.e., each tile has side length $3/12$. 
Now $B$ (and $B'$) must have an edge on the
border of the unit square, so the remaining
$7/12$ must be covered by tiles of side length
$3/12$, i.e., an integer multiple of $3/12$
must be equal to $7/12$.
This is impossible. Thus the optimal
tiling must either have $t\neq 5/12$ or it
must have 5 remaining tiles of different
sizes. In either case, the total side length
will be smaller than $8/3$. Hence
$\psi_3(8)<8/3$, as claimed.
\end{proof}

We will now proceed with the proof that
$\psi_3(8)=2.6$. Suppose $P$ is an 
optimal tiling, i.e., $\sigma(P)$ is
maximal. We know that $\sigma(P)\geq 2.6$.
As always, we assume without harm that
$P$ contains a corner tile $A$ with
a corner at $(0,0)$; there are also
at least two tiles $B$ and $B'$ with
edge lengths $s_B=s_{B'}=1-s_A$.
\begin{lemma}\label{atmost3}
There are at most 3 tiles with
edge lengths $s_B$.
\end{lemma}
\begin{proof}
Suppose there are 4 tiles with
edge lengths $t=s_B$. Then these
4 tiles, together with $A$, have
total area $4t^2+(1-t)^2$, leaving
an area of $2t-5t^2$ for the 
remaining 3 tiles. The edge lengths
of these 3 tiles sum up to
at most $\sqrt{3(2t-5t^2)}=
\sqrt{6t-15t^2}$, so the total edge
length of all 8 tiles is at most
$4t+(1-t)+\sqrt{6t-15t^2}=
1+3t+\sqrt{6t-15t^2}$. It is straightforward
to calculate that this function
has a maximum value of $\frac{8+2\sqrt{6}}{5}
<2.58$ (at $t=\frac{4+\sqrt{6}}{20}$).
Since $\sigma(P)\geq 2.6$, any tiling
with 4 tiles of edge length $s_B$ cannot
be optimal. The situation is even
worse if the tiling has more than 4 tiles
of edge length $s_B$.
\end{proof}

\begin{lemma}\label{exactly3}
In an optimal tiling, there are exactly 3
tiles with edge lengths $s_B$.
\end{lemma}
\begin{proof}
By Lemma~\ref{atmost3} we need to show
that there are at least 3 tiles with edge
lengths $s_B$. We already know 
that
there are 2 tiles, $B$ and $B'$, with
$s_{B'}=s_B$. Suppose there are
no other tiles with edge length $s_B$;
we will derive a contradiction.

Recall that $B$ can be assumed
to be a right coastal tile
with a corner at $(1,0)$
and that $B'$ can be assumed to
have a corner at
$(0,1)$. Thus we have the following configuration.
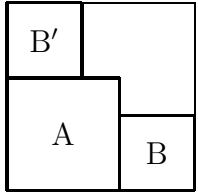
\begin{center}
\unitlength 0.5mm
\begin{picture}(50.00,50.00)(0,0)

\linethickness{0.15mm}
\put(0.00,0.00){\line(1,0){50.00}}
\put(0.00,0.00){\line(0,1){50.00}}
\put(50.00,0.00){\line(0,1){50.00}}
\put(0.00,50.00){\line(1,0){50.00}}


\linethickness{0.3mm}
\put(0.00,30.00){\line(1,0){20.00}}
\put(0,30.00){\line(0,1){20.00}}
\put(20.00,30.00){\line(0,1){20.00}}
\put(0.00,50.00){\line(1,0){20.00}}

\linethickness{0.3mm}
\put(30.00,0.00){\line(1,0){20.00}}
\put(30.00,0.00){\line(0,1){20.00}}
\put(50.00,0.00){\line(0,1){20.00}}
\put(30.00,20.00){\line(1,0){20.00}}

\linethickness{0.3mm}
\put(0.00,0.00){\line(1,0){30.00}}
\put(0.00,0.00){\line(0,1){30.00}}
\put(30.00,0.00){\line(0,1){30.00}}
\put(0.00,30.00){\line(1,0){30.00}}

\put(15.00,15.00){\makebox(0,0)[cc]{A}}


\put(10.00,40.00){\makebox(0,0)[cc]{B$'$}}

\put(40.00,10.00){\makebox(0,0)[cc]{B}}

\end{picture}
\end{center}

If $s_A=s_B=1/2$, then the remaining
empty square of size $1/2$-by-$1/2$ needs
to be tiled by 5 squares. This is impossible. 
It follows that $s_A>1/2$
(and so $s_B<1/2$). 

Let $C$ denote the tile with a corner
at $(1,1)$. 
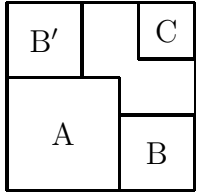
\begin{center}
\unitlength 0.5mm
\begin{picture}(50.00,50.00)(0,0)

\linethickness{0.15mm}
\put(0.00,0.00){\line(1,0){50.00}}
\put(0.00,0.00){\line(0,1){50.00}}
\put(50.00,0.00){\line(0,1){50.00}}
\put(0.00,50.00){\line(1,0){50.00}}

\linethickness{0.3mm}
\put(35.00,35.00){\line(1,0){15.00}}
\put(35.00,35.00){\line(0,1){15.00}}
\put(50.00,35.00){\line(0,1){15.00}}
\put(35.00,50.00){\line(1,0){15.00}}

\linethickness{0.3mm}
\put(0.00,30.00){\line(1,0){20.00}}
\put(0,30.00){\line(0,1){20.00}}
\put(20.00,30.00){\line(0,1){20.00}}
\put(0.00,50.00){\line(1,0){20.00}}

\linethickness{0.3mm}
\put(30.00,0.00){\line(1,0){20.00}}
\put(30.00,0.00){\line(0,1){20.00}}
\put(50.00,0.00){\line(0,1){20.00}}
\put(30.00,20.00){\line(1,0){20.00}}

\linethickness{0.3mm}
\put(0.00,0.00){\line(1,0){30.00}}
\put(0.00,0.00){\line(0,1){30.00}}
\put(30.00,0.00){\line(0,1){30.00}}
\put(0.00,30.00){\line(1,0){30.00}}

\put(15.00,15.00){\makebox(0,0)[cc]{A}}

\put(42.50,42.50){\makebox(0,0)[cc]{C}}

\put(10.00,40.00){\makebox(0,0)[cc]{B$'$}}

\put(40.00,10.00){\makebox(0,0)[cc]{B}}

\end{picture}
\end{center}

There are 4 tiles that remain to
be placed. At least 2 must share a border
on the line $y=s_B$ with
$B$ (if there were only 1, then it
must have edge length $s_B$); similarly,
at least 2 must share a border
on the line $x=s_B$ with $B'$.
Thus there are exactly 2 tiles on
top of $B$: one a right coastal tile
(call it $E$) and one an inland
tile, with a corner at $(s_A, s_B)$
(call it $D$). Similarly, there
are 2 tiles to the right of $B'$,
one on the north border (call it
$E'$) and one with a corner
at $(s_B,s_A)$ (call it $D'$).

Note that $s_E=s_{E'}=1-s_B-s_C$;
also, $s_D=s_{D'}=s_B-s_E$. 
Thus the tiling is symmetric
with respect to the main diagonal
$y=x$. 

Now consider the line connecting
the northwest corner of $A$ 
to the southeast corner of $C$.
Since this is a diagonal line, it
must intersect the interior
of a tile, either $D$ or $E$
or $D'$ or $E'$. But the symmetry
of the tiling indicates that
the aforementioned line must
intersect the interior of \emph{two}
tiles, contradicting our requirement
that the tiles don't overlap. 

\end{proof}

We note here a by-product of 
the proof: it cannot be the case
that $s_A=s_B=1/2$. Thus 
we have $s_A>s_B$.

Denote by $B''$ the third
tile whose side length is
equal to $s_B$. As above,
we can assume that $B''$
lies adjacent to $B'$. 

\begin{thm}
$\psi_3(8)=\frac{13}{5}$.
\end{thm}
\begin{proof}
Recall that $C$ is the tile
with a corner at $(1,1)$. We
have a configuration
similar to the following.

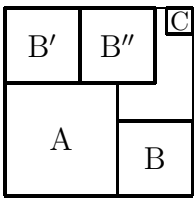
\begin{center}









\unitlength 0.5mm
\begin{picture}(50.00,50.00)(0,0)

\linethickness{0.15mm}
\put(0.00,0.00){\line(1,0){50.00}}
\put(0.00,0.00){\line(0,1){50.00}}
\put(50.00,0.00){\line(0,1){50.00}}
\put(0.00,50.00){\line(1,0){50.00}}

\linethickness{0.3mm}
\put(43.00,43.00){\line(1,0){7.00}}
\put(43.00,43.00){\line(0,1){7.00}}
\put(50.00,43.00){\line(0,1){7.00}}
\put(43.00,50.00){\line(1,0){7.00}}

\linethickness{0.3mm}
\put(0.00,30.00){\line(1,0){20.00}}
\put(0.00,30.00){\line(0,1){20.00}}
\put(20.00,30.00){\line(0,1){20.00}}
\put(0.00,50.00){\line(1,0){20.00}}

\linethickness{0.3mm}
\put(20.00,30.00){\line(1,0){20.00}}
\put(20.00,30.00){\line(0,1){20.00}}
\put(40.00,30.00){\line(0,1){20.00}}
\put(20.00,50.00){\line(1,0){20.00}}

\linethickness{0.3mm}
\put(30.00,0.00){\line(1,0){20.00}}
\put(30.00,0.00){\line(0,1){20.00}}
\put(50.00,0.00){\line(0,1){20.00}}
\put(30.00,20.00){\line(1,0){20.00}}

\linethickness{0.3mm}
\put(0.00,0.00){\line(1,0){30.00}}
\put(0.00,0.00){\line(0,1){30.00}}
\put(30.00,0.00){\line(0,1){30.00}}
\put(0.00,30.00){\line(1,0){30.00}}

\put(15.00,15.00){\makebox(0,0)[cc]{A}}

\put(46.50,46.50){\makebox(0,0)[cc]{\footnotesize{C}}}

\put(40.00,10.00){\makebox(0,0)[cc]{B}}

\put(10.00,40.00){\makebox(0,0)[cc]{B$'$}}

\put(30.00,40.00){\makebox(0,0)[cc]{B$''$}}

\end{picture}

\end{center}

Three tiles remain to be placed. Two of them
are right coastal tiles and one---the one
with a corner at $(s_A,s_B)$, which
as before we call $D$---is an inland tile.
Thus there are two inland tiles, $B''$ and $D$,
and the total edge length of this tiling is
$2+s_{B}+s_D$. 

If $s_D>\frac12 s_B$, then the two
right coastal tiles besides $B$ and $C$
must each have edge length $s_B-s_D
<\frac12 s_B$; thus $\sigma(\C_1)<
s_B+\frac12 s_B+\frac12 s_B + s_C=
2s_B+s_C$. But looking at the north
border we see that $2s_B+s_C\leq 1$,
so $\sigma(\C_1)<1$, a contradiction.
Thus we must have $s_D\leq \frac12s_B$,
so the total length of the tiling is at
most $2+s_B+\frac12 s_B=2+\frac32 s_B$.
Therefore $2+\frac32 s_B\geq \frac{13}{5}$,
i.e., $s_B\geq \frac25$.

Now consider just the tiles $A$, $B$, $B'$,
and $B''$. Let $t=s_B$ as before.
These tiles have total area
$3t^2+(1-t)^2$, leaving an area of
$2t-4t^2$ to be covered with 4 tiles.
By Lemma~\ref{CS},
the total edge lengths of these 4 tiles
is at most $\sqrt{4(2t-4t^2)}$; thus
the total edge length of all tiles
is at most $3t+(1-t)+\sqrt{4(2t-4t^2)}$.
For $t\geq \frac25$, this function
has a maximum value of $\frac{13}{5}$
(which occurs at $t=\frac25$); thus
$\sigma(\C)\leq \frac{13}{5}$, as required.

\end{proof}

\textbf{Note}: I do not know the value of $\psi_3(k^2-1)$
for $k>3$. It is possible to show that $\psi_3(k^2-1)\geq
k-\frac{1}{k-1}$ as follows. Start with a standard
$(k+1)\times (k+1)$ tiling, and replace a $k\times k$
subsquare with a standard $(k-1)\times (k-1)$ tiling. We
now have a tiling with $(k+1)^2-k^2+(k-1)^2 = k^2+2$
tiles. There are (a) $2k+1$ tiles with edge length
$1/(k+1)$, and (b) $(k-1)^2$ tiles with edge length
$\frac{k}{k^2-1}$. Pick any $2\times 2$ subsquare in (b)
and replace
it with one big square; we now have a tiling with
$k^2-1$ tiles. The total edge length of this tiling
is $\frac{2k+1}{k+1} + \frac{k(k-1)^2}{k^2-1} - \frac{2k}{k^2-1}
= k-\frac{1}{k-1}$.

\end{document}